\documentclass[a4paper]{amsart}
\usepackage[utf8]{inputenc}
\usepackage[T1]{fontenc}

\usepackage{geometry}
\usepackage{amssymb}
\usepackage{amsfonts}
\usepackage{amsthm}
\usepackage{calrsfs}
\usepackage{array}
\usepackage{bbm}
\usepackage{stmaryrd}
\usepackage{hyperref}
\usepackage{enumerate}
\usepackage{mathabx}
\usepackage{enumitem}
\usepackage[demo]{graphicx}
\usepackage{caption}
\usepackage{subcaption}

\usepackage{csquotes}

\usepackage{mathtools}
\mathtoolsset{showonlyrefs=true}
\usepackage[svgnames]{xcolor}
\usepackage{hyperref}

\usepackage{tikz}
\usetikzlibrary{positioning}

\newcommand{\R}{\mathbb{R}}

\newcommand{\jap}[1]{\langle #1 \rangle}

\newcommand{\fia}{\mathbbm{1}_{|\Phi-\alpha|<M}}

\theoremstyle{plain}
\newtheorem{thm}{Theorem}
\newtheorem*{thm*}{Theorem}

\newtheorem{cor}[thm]{Corollary}
\newtheorem{lem}[thm]{Lemma}

\theoremstyle{definition}

\theoremstyle{remark}
\newtheorem{nb}[thm]{Remark}

\numberwithin{equation}{section}
\newtheoremstyle{mytheoremstyle} 
{\topsep}                    
{\topsep}                    
{}                   
{}                           
{\scshape}                   
{.}                          
{.5em}                       
{}  

\theoremstyle{mytheoremstyle} 
\theoremstyle{mytheoremstyle} 

\makeatletter
\@namedef{subjclassname@2020}{%
	\textup{2020} Mathematics Subject Classification}
\makeatother

\date{}
\author{Simão Correia}

\title{Improved global well-posedness for the quartic Korteweg-de Vries equation}


\thanks{S. C. was partially supported by Funda\c{c}\~ao para a Ci\^encia e Tecnologia, through CAMGSD, IST-ID
	(projects UIDB/04459/2020 and UIDP/04459/2020) and through the project NoDES (PTDC/MAT-PUR/1788/2020)}

\allowdisplaybreaks

\begin{document}
\maketitle
\begin{abstract}
	We prove that the quartic Korteweg-de Vries equation is globally well-posed for real-valued initial data in $H^s(\R)$, $s>-1/24$.
\end{abstract}
\section{Introduction}
In this note, we consider the quartic Korteweg-de Vries equation on the line
\begin{equation}\label{4kdv}\tag{4KdV}
	u_t + u_{xxx} + (u^4)_x=0
\end{equation}
with real-valued initial data $u_0\in H^s(\R\to \R)$. As it was proven by Grünrock in \cite{grunrock4kdv}, the initial value problem is locally well-posed for $s>-1/6$, which is the scaling-subcritical range. Later on, Tao \cite{tao4kdv} extended this result to the critical regularity $\dot{H}^{-1/6}(\R)$ and proved global existence for small data.

For large real-valued initial data in $H^s$, $s\ge 0$, the conservation of the $L^2$ norm,
$$
\|u(t)\|_{L^2}=\|u_0\|_{L^2},
$$ automatically gives a global well-posedness result. For lower regularities, Grünrock, Panthee and Silva \cite{GPS} applied the first iteration of the $I$-method, developed by Colliander, Keel, Stafillani, Takaoka and Tao \cite{ckstt_kdv, ckstt_dnls, ckstt_dnls2, ckstt_kdv2}, to extend this global existence result up to $s>-1/42$, which is the best result known-to-date.

An improvement on the global existence regularity through the \textit{I}-method requires a second modification of the energy functional, which is a very nontrivial problem. Apart from some cases where the algebraic structure of the dispersion/nonlinearity is
relatively simple, it is an open problem how to further iterate the \textit{I}-method for general dispersive equations.

The main purpose of this work is to present a general strategy of defining second iterations of the \textit{I}-method and to identify the precise mechanisms that allow for an improvement on the global well-posedness regularity threshold. In particular, we will prove
\begin{thm}\label{thm:main}
	The Cauchy problem for \eqref{4kdv} is globally well-posed in $ H^s(\R\to \R)$ for $s>-1/24$. Moreover, given $T>0$,
	$$
	\sup_{t\in[0,T]} \|u(t)\|_{H^s}\lesssim_\eta (1+T)^{\eta}\|u_0\|_{H^s},\quad \mbox{ for any }\ \eta>\frac{-2s}{1+24s}.
	$$
\end{thm}

\section{Description of the method}\label{sec:descript}
We follow closely the \textit{I}-method as introduced in \cite{ckstt_kdv2}. Given $N\gg 1$, define the multiplier
$$
m(\xi)=\begin{cases}
	1,& |\xi|<N\\ \left(\frac{|\xi|}{N}\right)^s, & |\xi|>N
\end{cases}
$$
and $\widehat{Iu}(\xi)=m(\xi)\hat{u}(\xi)$ (where $\ \widehat{\cdot}\ $ denotes the Fourier transform). Observe that, for $N$ fixed,
$$
\|Iu\|_{L^2} \simeq \|u\|_{H^s},
$$
so that the $I$-operator improves the regularity up to the energy regularity $s=0$.

Given $n\ge 2$, a multiplier $M_n: \R^n\to \R$ and functions $f_1,\dots, f_n\in \mathcal{S}'(\R)$, set 
$$
\Lambda_n\left(M_n;f_1,\dots, f_n\right)=\int_{\Gamma} M_n(\xi_1,\dots,\xi_n)\prod_{j=1}^n\widehat{f_j}(\xi_j)d\Xi,\qquad \Gamma=\left\{ (\xi_1,\dots,\xi_n)\in \R^n : \sum_{j=1}^{n}\xi_j=0 \right\}.
$$
If $f_j\equiv f$, we write simply $\Lambda_n(M_n;f)$. We will use the notation $\xi_{ijk}=\xi_i+\xi_j+\xi_k$ (and analogously for any finite sum of frequencies). Moreover, we shall abbreviate $m(\xi_\ast)$ (resp. $\hat{u}(\xi_\ast)$) as $m_\ast$ (resp. $\hat{u}_\ast$). The next lemma follows from the direct integration of \eqref{4kdv} (see \cite{ckstt_kdv2,corchopanthee_mkdv}):

\begin{lem}\label{lem:derivLambda}
	If $u$ is a smooth solution to \eqref{4kdv} and $M_n$ is a symmetric multiplier, then
	$$
	\frac{d}{dt}\Lambda_n(M_n;u)=i\Lambda_n\left(M_n\Phi_n;u\right)-in\Lambda_{n+3}(M_n(\xi_1,\dots,\xi_{n-1},\xi_n+\dots+\xi_{n+3})(\xi_n+\dots+\xi_{n+3});u).
	$$
	where $\Phi_n=\xi_1^3+\dots + \xi_n^3$.
\end{lem}

According to the $I$-method, the first modified energy is simply $E_1=\|Iu\|_{L^2}^2=\Lambda_2(m_1m_2;u)$. By Lemma \ref{lem:derivLambda} and the convolution relation,
$$
\frac{dE_1}{dt}=-2i\Lambda_5\left(m_1m_{2345}\xi_{2345};u\right) = \frac{2i}{5}\Lambda_5\left(m_1^2\xi_1 + \dots + m_5^2\xi_5; u \right).
$$
and thus
\begin{equation}\label{eq:derivadaE1}
	E_1(u(t))-E_1(u(0))=\frac{2i}{5}\int_0^t \Lambda_5\left(m_1^2\xi_1 + \dots + m_5^2\xi_5; u \right) ds.
\end{equation}
In order to control the right-hand side, we introduce the $X^{s,b}$ norm
$$
\|u\|_{X^{s,b}}^2 = \int_\R \int_\R \jap{\xi}^{2s} \jap{\tau-\xi^3}^{2b}|\mathcal{F}_{t,x}u|^2(\tau,\xi) d\tau d\xi,
$$
where $\mathcal{F}_{t,x}$ denotes the Fourier transform in both variables.
\medskip

Applying Cauchy-Schwarz (see Lemma \ref{lem:cs} below), given\footnote{Given $a\in\R$, $a^+$ (resp. $a^-$) represents a number slightly larger (resp. smaller) than $a$.} $b=\frac{1}{2}^+$, $b'=(b-1)^+$,
$$
\left|\int_0^t \Lambda_5\left(m_1^2\xi_1 + \dots + m_5^2\xi_5; u \right) ds \right| \lesssim t^{0^+}\left\| \left(\frac{m_1^2\xi_1+\dots+m_4^2\xi_4-m_{1234}^2\xi_{1234}}{m_{1234}}\right)^\vee u_1u_2u_3u_4 \right\|_{X^{0,b'}}\|Iu\|_{X^{0,b}},
$$
In \cite{GPS}, the authors proved the estimate
\begin{equation}\label{eq:N12_0}
\left\| \left(\frac{m_1^2\xi_1+\dots+m_4^2\xi_4-m_{1234}^2\xi_{1234}}{m_{1234}}\right)^\vee u_1u_2u_3u_4 \right\|_{X^{0,b'}}\lesssim N^{(-1/2)^+}\prod_{j=1}^4\|Iu_j\|_{X^{0,b}}.
\end{equation}
Even though their proof was direct, one can actually see \eqref{eq:N12_0} as
%
a consequence  of the nonlinear smoothing effect proved in \cite[Proposition 6]{COS23},
\begin{equation}\label{eq:nonlinearsmooth0}
	\||\nabla|\jap{\nabla}^{(1/2)^-}u_1u_2u_3u_4\|_{X^{0,b'}}\lesssim \prod_{j=1}^4\|u_j\|_{X^{0,b}}.
\end{equation}
(see Corollary \ref{cor:gps}). From \eqref{eq:N12_0}, we conclude that
$$
|E_1(u(t))-E_1(u(0))| \lesssim t^{0^+}N^{(-1/2)^+}\|Iu\|_{X^{0,b}}^5.
$$
Applying the general scaling arguments of the $I$-method, this then yields global well-posedness for $s>-1/42$. More precisely, it is the rate of decay in $N$ that ultimately determines the regularity threshold. In order to lower the threshold to $s>-1/24$, we need to define a  modified energy $E_2$ (equivalent to $E_1$) for which
$$
|E_2(u(t))-E_2(u(0))| \lesssim t^{0^+}N^{-1^-}.
$$
Let us look for a modified energy of the form
$$
E_2=E_1 + \Lambda_5(M_5;u), \quad \mbox{for some symmetric multiplier }M_5.
$$
Then
$$
\frac{dE_2}{dt} = i\Lambda_5\left( \frac{2\left(m_1^2\xi_1+\dots +m_5^2\xi_5\right)}{5} + M_5\Phi_5;u \right) - 5i\Lambda_8\left(M_5(\xi_1,\dots,\xi_4,\xi_{5678})\xi_{5678} ;u\right)
$$
As one expects higher-order terms to contribute with stronger decays, the most standard choice is to choose $M_5$ so that the quintic multiplier is identically zero:
\begin{equation}\label{eq:tentativaM5}
	M_5=M'_5:=-\frac{2\left(m_1^2\xi_1+\dots +m_5^2\xi_5\right)}{5\Phi_5}.
\end{equation}
However, one must ensure that $M_5$ is not a singular multiplier. For this to happen, an algebraic "miracle" must occur, so that the zero set of the denominator is completely included in the zero set of the numerator. This miracle \textit{does occur} for the KdV, the mKdV and the derivative NLS, which lead to improved global existence results \cite{ckstt_dnls2,ckstt_kdv2,  corchopanthee_mkdv}. Unfortunately, such a property is far from being valid in general \cite{farah_5kdv,  GPS, kishimoto_zakharov, tzirakis_5nls}. In our context, one can actually check that $M_5$ as defined in \eqref{eq:tentativaM5} is singular.

In face of the singularity of $M_5$, one can try to be less ambitious: instead of a total cancellation of the quintic multiplier, one can aim for a \textit{partial} cancellation which still eliminates the worst contributions. Given a symmetric set $D$ on which $M_5'$ does not have singularities, one may set
\begin{equation}\label{eq:tentativaM52}
	M_5=M_5'\mathbbm{1}_D.
\end{equation}
In this way, $E_2$ is well-defined and
$$
\frac{dE_2}{dt} = i\Lambda_5\left( \frac{2\left(m_1^2\xi_1+\dots m_5^2\xi_5\right)}{5}\mathbbm{1}_D^c;u \right) - 5i\Lambda_8\left(M_5(\xi_1,\dots,\xi_4,\xi_{5678})\xi_{5678} ;u\right)
$$
The larger the set $D$, the better the cancellation and the subsequent decay in $N$. Currently, this is the most advanced form of the $I$-method and it has been sucessfully applied in \cite{ ckstt_2d_nls, kishimoto_zakharov}. However, the choice of $D$ is far from being trivial and
one often must rely on a precise knowledge of the resonant set $\{\Phi_5=0\}$ in order to derive appropriate estimates in $X^{s,b}$ spaces. 

In the recent work \cite{COS23}, the authors established the connection between multilinear estimates in Bourgain spaces and frequency-restricted estimates, which are estimates in spatial frequency over sublevel sets of the resonance function $\Phi_n$. Moreover, the derivation of such frequency-restricted estimates relies not on the precise description of the sublevel set, but rather on its \textit{geometry}. In particular, frequency-restricted estimates are perfectly suited to derive bounds on a second iteration of the energy, as the choice $D=\{|\Phi_5|>K\}$, for some $0<K<1$ to be chosen, induces the restriction to a sublevel of the resonance function.

We therefore consider the second modified energy as
$$
E_2=\Lambda_2(m_1m_2;u)-\Lambda_5\left(\frac{2\left(m_1^2\xi_1+\dots+m_5^2\xi_5\right)}{5\Phi_5}\mathbbm{1}_{|\Phi_5|>K};u\right).
$$
\begin{nb}
	We remark that $E_2$ is nothing but the original energy $E_1$ after an appropriate normal form reduction (see \cite{Bou_normal}). Indeed, let us go back to \eqref{eq:derivadaE1} and write the r.h.s. in terms of the profile $\tilde{u}=e^{-it\xi^3}\hat{u}$:
	$$
	E_1(u(t))-E_1(u(0))=\frac{2i}{5}\int_0^t\int_\Gamma e^{is\Phi_5} \left(m_1^2\xi_1 + \dots + m_5^2\xi_5\right) \tilde{u}_1\dots \tilde{u}_5 d\Xi ds.
	$$
	In order to exploit the oscillations in time, one can integrate by parts in time using the relation $$e^{is\Phi_5}=\frac{1}{i\Phi_5}\partial_s(e^{is\Phi_5}).$$ However, since the resonance function can become singular, one must first decompose the integral into the near-resonant region $|\Phi_5|<K$ and the nonresonant region $|\Phi_5|>K$ and perform the integration by parts only in the second. Doing so produces the boundary term
	$$
	\left[ \int_\Gamma e^{is\Phi_5} \frac{m_1^2\xi_1 + \dots + m_5^2\xi_5}{i\Phi_5}\mathbbm{1}_{|\Phi_5|>K} \tilde{u}_1\dots \tilde{u}_5 d\Xi  \right]_{s=0}^{s=t}
	$$
	which is nothing more than the quintic term introduced in the second modified energy.
\end{nb}

\bigskip
Having deduced the candidate for the second modified energy, we must check that $E_2$ is well-defined whenever $Iu\in L^2$:
\begin{lem}[$E_2$ is well-defined]\label{lem:defiE2}
	Given $u\in \mathcal{S}'(\R)$ such that $Iu\in L^2(\R)$,
	\begin{equation}
		\left|\Lambda_5\left(\frac{m_1^2\xi_1+\dots+m_5^2\xi_5}{\Phi_5}\mathbbm{1}_{|\Phi_5|>K};u\right)\right|\lesssim N^{(-1/2)^+}K^{0^-}\|Iu\|_{L^2}^5.
	\end{equation}
\end{lem}
The time derivative for the second modified energy is given by
$$
\frac{dE_2}{dt} = 2i\Lambda_5\left( \frac{m_1^2\xi_1+\dots m_5^2\xi_5}{5}\mathbbm{1}_{|\Phi_5|<K};u \right) - 5i\Lambda_8\left(\frac{m_1^2\xi_1+\dots+m_4^2\xi_4-m_{1234}^2\xi_{1234}}{\Phi_5}\xi_{5678}\mathbbm{1}_{|\Phi_5|>K};u \right).
$$
The first term is also well-behaved:
\begin{lem}[Control of the $\Lambda_5$-term]\label{lem:L5}
	Given $u\in\mathcal{S}'(\R^2)$ with $Iu\in X^{0,(1/2)^+}$,
	\begin{equation}
		\left|	\int_0^t \Lambda_5\left( \frac{m_1^2\xi_1+\dots +m_5^2\xi_5}{5}\mathbbm{1}_{|\Phi_5|<K};u \right) ds \right| \lesssim t^{0^+}N^{(-1/2)^+}K^{1^-}\|Iu\|_{X^{0,1/2^+}}^5.
	\end{equation}
\end{lem}

As observed in \cite{COS23}, the restriction to the sublevel sets of the resonance function is a natural constraint for multilinear estimates in Bourgain spaces and the two lemmas above will be seen as direct consequences of the nonlinear smoothing effect \eqref{eq:nonlinearsmooth0}. 

\bigskip

We now focus on the $\Lambda_8$-term and give an heuristic argument for a stronger decay in $N$. First, let us recall the origin of this term. As one differentiates the quintic term of the energy, the time derivative falls onto the various instances of $u$. Using \eqref{4kdv}, the $\Lambda_8$-term appears as one replaces the $u_t$ factor with the nonlinearity $\partial_x(u^4)$. Let us represent the process with tree graphs\footnote{This representation is inspired in the normal form reduction procedure, see \cite{guo, koy}.}: we identify the nonlinear term $\partial_x(u^4)$ with
\begin{equation}\label{eq:tree1}
	\begin{tikzpicture}
		
		\node {$\xi_1$} [sibling distance = 1.5cm]
		child {node {$\xi_2$}} 
		child {node {$\xi_3$}} 
		child {node {$\xi_4$}} 
		child {node {$\xi_5$}
		};
	\end{tikzpicture}
\end{equation}
and the $\Lambda_8$-term with
\begin{equation}\label{eq:tree2}
\begin{tikzpicture}
	\node {$\xi_1$} [sibling distance = 1.5cm]
	child {node {$\xi_2$}} 
	child {node {$\xi_3$}} 
	child {node {$\xi_4$}} 
	child {node {$\xi_{5678}$}
		child {node {$\xi_5$}}
		child {node {$\xi_6$}}
		child {node {$\xi_7$}}
		child {node {$\xi_8$}}
	};
\end{tikzpicture}
\end{equation}
By the nonlinear smoothing effect \eqref{eq:nonlinearsmooth0}, the regularity of the parent node in \eqref{eq:tree1} can be upgraded by (almost) $1/2$-derivatives.

Let us now focus on the tree \eqref{eq:tree2}. First, if all frequencies are $\ll N$, then the multiplier vanishes (this comes from the conservation law at the $L^2$ level). Therefore we can assume that two frequencies are $\gtrsim N$. In the worst-case scenario, one of them is $\xi_1$.

If $|\xi_{5678}|\gtrsim N$, then we can use the nonlinear smoothing effect twice, once for each subtree. Since the parent of each subtree is $\gtrsim N$, one can gain $N^{(-1/2)^+}$ from each estimate and reach $N^{(-1)^+}$. The case where one of the frequencies $\xi_5,\dots,\xi_8$ is $\gtrsim N$ can be handled in the same way. 

We now suppose that $|\xi_5|,\dots,|\xi_8|\ll N$. Again by the smoothing effect, we know that the regularity of the $\xi_{5678}$-node is $(1/2)^-$. One can then ask if this regularity can be transfered to $\xi_1$, giving a gain of a full derivative. This requires a refined nonlinear smoothing estimate (Lemma \ref{lem:refined}), and it does hold in the complement of the region
$$
\{|\xi_1|\simeq \dots \simeq |\xi_4| \gg |\xi_{5678}|\}.
$$
In this region, the resonance function $\Phi_5$ is stationary in all frequencies but $\xi_{5678}$, which prevents us from using the extra regularity in this variable. The key point is to use the full tree \eqref{eq:tree2}, as one "gains room" (in terms of number of spatial frequencies) in order to exploit both the nonstationarity of $\Phi_5$ in $\xi_{5678}$ and the smoothing effect at the lowest frequencies (see Lemma \ref{lem:D3}). 

\bigskip

In conclusion, the nonlinear smoothing effect gives a gain of $N^{(-1)^+}$ except in a very particular region, where a full 8-linear smoothing estimate must be derived.
\begin{lem}[Control of the $\Lambda_8$-term]\label{lem:L8}
	Given $u\in\mathcal{S}'(\R^2)$ with $Iu\in X^{0,(1/2)^+}$,
	\begin{equation}
		\left| \int_0^t \Lambda_8\left(\frac{m_1^2\xi_1+\dots+m_4^2\xi_4-m_{1234}^2\xi_{1234}}{\Phi_5}\xi_{5678}\mathbbm{1}_{|\Phi_5|>K};u \right) ds \right|\lesssim t^{0^+}N^{(-1)^+}K^{0^-}\|Iu\|_{X^{0,(1/2)^+}}^8.
	\end{equation}
\end{lem}
Now we optimize the estimates in $K$ by choosing $K=N^{-1/2}$. For this specific choice, we conclude that
\begin{equation}\label{eq:almosconserved}
	|E_2(u(t))-E_2(u(0))|\lesssim t^{0^+}N^{(-1)^+}\left(\|Iu\|_{X^{0,(1/2)^+}}^5 + \|Iu\|_{X^{0,(1/2)^+}}^8\right)
\end{equation}
From here, the proof of Theorem \ref{thm:main} now follows from standard arguments (see, for example, \cite[Proof of Theorem 1]{GPS}). We include the proof for the sake of completeness.

\begin{nb}
	As we will see, the necessary nonlinear smoothing estimates follow from specific frequency-restricted estimates. While it could be possible to prove the first without resorting to the latter, we claim tha frequency-restricted estimates appear naturally, since  the multipliers in the second modified energy already involve the restriction to certain sublevel sets of the resonance function.
\end{nb}

The rest of this paper is organized as follows. In Section \ref{sec:multi1}, we reduce Lemmas \ref{lem:defiE2} and \ref{lem:L5} to the nonlinear smoothing effect. In Section \ref{sec:multi2}, we prove Lemma \ref{lem:L8}, following the strategy described above. In Section \ref{sec:thm}, we prove the main result.

\bigskip
\textbf{Acknowledgements.}
I would like to thank Jorge Silva for the many fruitful and enthusiastic discussions on the subject.

\section{Multilinear estimates}

\subsection{Preliminaries and proof of Lemmas \ref{lem:defiE2} and \ref{lem:L5}}\label{sec:multi1}
We begin with a standard monotonicity result, which follows directly from a duality argument.

\begin{lem}\label{lem:monot}
	Suppose that, for a certain positive multiplier ${M}={M}(\xi_1,\dots,\xi_k)$,
	$$
	\left\| {M}^\vee u_1\dots u_k \right\|_{X^{s,b}}\le C\prod_{j=1}^k \|u_j\|_{X^{s_j,b_j}}
	$$
	and that ${N}$ is another multiplier with $|{N}|\le {M}$. Then
	$$
	\left\| {N}^\vee u_1\dots u_k \right\|_{X^{s,b}}\le C\prod_{j=1}^k \|u_j\|_{X^{s_j,b_j}}
	$$
\end{lem}

From now on, $b=1/2^+$ and $b'=(b-1)^+$.
\begin{lem}\label{lem:cs}
	Given $u_1,u_2\in\mathcal{S}(\R^2)$ and $t\in \R$,
	$$
	\left|\int_0^t \int u_1(s,x)\overline{u_2}(s,x) dx ds\right|\lesssim t^{0^+}\|u_1\|_{X^{0,b'}}\|u_2\|_{X^{0,b}}.
	$$
\end{lem}
\begin{proof}
	We have
\begin{align*}
	&\ \int_0^t \int u_1(s,x)\overline{u_2}(s,x) dx ds\\ =& \	\int_0^t \int (e^{-t\partial_x^3}u_1)(s,x)\overline{e^{-t\partial_x^3}u_2}(s,x) dx ds \\=&\ \int\int_{\tau=\tau_1+\tau_2} \frac{e^{it\tau}-1}{i\tau}\left(\mathcal{F}_{t,x}e^{-t\partial_x^3}u_1\right)(\tau_1,\xi)\left(\overline{\mathcal{F}_{t,x}e^{-t\partial_x^3}u_2}\right)(\tau_2,\xi) d\tau_1d\tau_2d\xi \\=&\ \int\int_{\tau=\tau_1+\tau_2} K(\tau_1,\tau_2)w_1(\tau_1,\xi)w_2(\tau_2,\xi) d\tau_1d\tau_2 d\xi,
\end{align*}
where
$$K= \frac{e^{it\tau}-1}{i\tau}\frac{1}{\jap{\tau_1}^{b'}\jap{\tau_2}^b}$$
$$ w_1(\tau,\xi)=\jap{\tau}^{b'}\left(\mathcal{F}_{t,x}e^{-t\partial_x^3}u_1\right)(\tau,\xi),\quad w_2(\tau,\xi)=\jap{\tau}^b\left(\overline{\mathcal{F}_{t,x}e^{-t\partial_x^3}u_2}\right)(\tau,\xi).$$
Let us check that $K$ is a Hilbert-Schmidt kernel in the $\tau$ variables. First, since $b+b'>0$ and
$$
\left|\frac{e^{it\tau}-1}{i\tau}\right|\lesssim \frac{t^{0^+}}{|\tau|^{1^-}},
$$
we have
$$
\sup_{\tau_1}\int |K(\tau_1,\tau_2)|d\tau_2 \lesssim  \sup_{\tau_1}\int \frac{t^{0^+}}{|\tau_1+\tau_2|^{1^-}\jap{\tau_1}^{b'}\jap{\tau_2}^b}d\tau_2 \lesssim t^{0^+}
$$
On the other hand, if $|\tau_1|\lesssim |\tau_1+\tau_2|$,
$$
\sup_{\tau_2}\int |K(\tau_1,\tau_2)|d\tau_1 \lesssim  \sup_{\tau_2}\int \frac{t^{0^+}}{|\tau_1+\tau_2|^{1^-}\jap{\tau_1}^{b'}\jap{\tau_2}^b}d\tau_1 \lesssim \sup_{\tau_2}\int \frac{t^{0^+}}{|\tau_1|^{1^-}\jap{\tau_1}^{b'}\jap{\tau_2}^b}d\tau_1 \lesssim t^{0^+},
$$
and, if  $|\tau_1|\gg |\tau_1+\tau_2|$,
$$
\sup_{\tau_2}\int |K(\tau_1,\tau_2)|d\tau_1 \lesssim  \sup_{\tau_2}\int \frac{t^{0^+}}{\jap{\tau_1+\tau_2}\jap{\tau_1}^{b'}\jap{\tau_2}^b}d\tau_1 \lesssim \sup_{\tau_2}\int \frac{t^{0^+}}{|\tau_1+\tau_2|^{1^-}\jap{\tau_1}^{b+b'}}d\tau_1 \lesssim t^{0^+},
$$
Therefore, by Schur's test, $K$ is a Hilbert-Schmidt kernel and thus
\begin{align*}
\left|\int\int_{\tau=\tau_1+\tau_2} K(\tau_1,\tau_2)w_1(\tau_1,\xi)w_2(\tau_2,\xi) d\tau_1d\tau_2 d\xi\right|&\lesssim t^{0^+}\int \|w_1(\xi)\|_{L^2_{\tau}}\|w_2(\xi)\|_{L^2_{\tau}}d\xi \\&\lesssim t^{0^+}\|w_1\|_{L^2_{\tau,\xi}}\|w_2\|_{L^2_{\tau,\xi}} = t^{0^+}\|u_1\|_{X^{0,b'}}\|u_2\|_{X^{0,b}}.
\end{align*}
\end{proof}

For $k>1$ fixed and given $\xi\in \R$, define
$$
\Gamma_\xi = \{(\xi_1,\dots,\xi_k)\in \R^k: \xi=\xi_1+\dots+\xi_k\}
$$
and 
$$
\Phi=\xi^3-\sum_{j=1}^k \xi_j^3.
$$
To simplify some notations, we write $\xi=\xi_0$ and $\tau=\tau_0$. Moreover, if $A\subset \{0,\dots, k\}$ is a set of indices, we abbreviate ``$\xi_j, \ j\in A$'' as ``$\xi_{j\in A}$''.
\begin{lem}[Reduction to frequency-restricted estimates]\label{lem:interpol}
	Let $M={M}(\xi_1,\dots,\xi_k)$ be a positive multiplier and $s, s_j\in\R$. Suppose that there exist $\emptyset\neq A \subsetneq \{0,\dots,k\}$, ${M}_n={M}_n(\xi_1,\dots,\xi_k)\ge 0$, $n=1,2$, such that
	$$
	\left({M}_1{M}_2\right)^{\frac{1}{2}}=\frac{M\jap{\xi}^{s'}}{\prod_{j=1}^k\jap{\xi_j}^{s_j}}
	$$
	and, for any $K>0$,
	\begin{equation}
		\sup_{\xi_{j\notin A},\alpha} \int_{\Gamma_\xi}{M}_1\fia d\xi_{j\in A} + 	\sup_{\xi_{j\in A},\alpha} \int_{\Gamma_\xi}{M}_2\fia d\xi_{j\notin A}  \lesssim K^{1^-}.
	\end{equation}
	Then 
\begin{equation}\label{eq:multibourg}
		\left\| {M}^\vee v_1\dots v_k\right\|_{X^{s,b'}}\lesssim \prod_{j=1}^k\|v_j\|_{X^{s_j,b}}
\end{equation}
	\begin{equation}\label{eq:multibourgKgrande}
		\left\| \left({M}\mathbbm{1}_{|\Phi|>K}\right)^\vee v_1\dots v_k\right\|_{X^{s,b'}}\lesssim K^{0^-}\prod_{j=1}^k\|v_j\|_{X^{s_j,b}}
	\end{equation}
\begin{equation}\label{eq:multibourgKpeq}
		\left\| \left({M}\mathbbm{1}_{|\Phi|<K}\right)^\vee v_1\dots v_k\right\|_{X^{s,b'}}\lesssim K^{1^-}\prod_{j=1}^k\|v_j\|_{X^{s_j,b}}
\end{equation}
\begin{equation}\label{eq:multibourgPhi}
	\left\| \left(\frac{{M}\mathbbm{1}_{|\Phi|>K}}{|\Phi|}\right)^\vee v_1\dots v_k\right\|_{X^{s,b'}}\lesssim K^{0^-}\|v_1\|_{X^{s_k,b'}}\prod_{j=2}^{k}\|v_j\|_{X^{s_j,b}}.
\end{equation}
	and
\begin{equation}\label{eq:multibourgPhi2}
		\left\| \left(\frac{{M}\mathbbm{1}_{|\Phi|>K}}{|\Phi|}\right)^\vee v_1\dots v_k\right\|_{X^{s,b}}\lesssim K^{0^-}\prod_{j=1}^{k}\|v_j\|_{X^{s_j,b}}.
\end{equation}
\end{lem}
\begin{nb}
	The first estimate was proved in \cite[Lemma 2]{COS23}. The last four estimates lead us to conclude that restrictions to sublevel sets of the resonance function $\Phi$ are compatible with multilinear estimates in Bourgain spaces.
\end{nb}
\begin{proof}
	By symmetry, we can suppose that $0\in A$. We begin with \eqref{eq:multibourg}. By duality, we want to prove that
\begin{equation}\label{eq:versaoL2}
		\left|\int_{\Gamma_\xi, \Gamma_\tau} \frac{M\jap{\xi}^s\jap{\tau-\xi^3}^{b'}}{\prod_{j=1}^k\jap{\xi_j}^{s_j}\jap{\tau_j-\xi_j^3}^b}w_1\dots w_k w \right|\lesssim \prod_{j=1}^k \|w_j\|_{L^2}
\end{equation}
	We argue by interpolation. First, suppose that $w_{j\in A}\in L^\infty$ and $w_{j\notin A}\in L^1$. Then
	\begin{align}
			&\left|\int_{\Gamma_\xi, \Gamma_\tau} \frac{M_1\jap{\tau-\xi^3}^{2b'}}{\prod_{j\in A, j\neq 0}\jap{\tau_j-\xi_j^3}^{2b}}w_1\dots w_k w \right| \\\lesssim\ & \left(	\sup_{\xi_{j\notin A},\tau_{j\notin A}}\int_{\Gamma_\xi, \Gamma_\tau} \frac{M_1\jap{\tau-\xi^3}^{2b'}}{\prod_{j=1}^k\jap{\tau_j-\xi_j^3}^{2b}} d\xi_{j\in A} d\tau_{j\in A}\right)\left(\prod_{j\in A}\|w_j\|_{L^\infty}\right)\left(\prod_{j\notin A}\|w_j\|_{L^1}\right)\label{eq:interp3}
	\end{align}
	Integrating first in the $\tau$ variables,
	\begin{align}
		\int_{\Gamma_\tau} \frac{\jap{\tau-\xi^3}^{2b'}}{\prod_{j=1}^k\jap{\tau_j-\xi_j^3}^{2b}} d\tau_{j\in A} \lesssim \frac{1}{\jap{\Phi - \sum_{j\notin A}(\tau_j-\xi_j^3)}^{-2b'}}.\label{eq:integraltau}
	\end{align}
	Setting $\alpha=\sum_{j\notin A}(\tau_j-\xi_j^3)$,
	\begin{align}
		\sup_{\xi_{j\notin A},\tau_{j\notin A}}\int_{\Gamma_\xi, \Gamma_\tau} \frac{M_1\jap{\tau-\xi^3}^{2b'}}{\prod_{j=1}^k\jap{\tau_j-\xi_j^3}^{2b}} d\xi_{j\in A} d\tau_{j\in A} &\lesssim 	\sup_{\xi_{j\notin A},\alpha}\int_{\Gamma_\xi} \frac{M_1}{\jap{\Phi-\alpha}^{-2b'}} d\xi_{j\in A}\label{eq:diadic} \\&\lesssim 	\sup_{\xi_{j\notin A},\alpha} \sum_{K {dyadic}} \frac{1}{K^{-2b'}}\int_{\Gamma_\xi} M_1\mathbbm{1}_{|\Phi-\alpha|<K} d\xi_{j\in A} \\&\lesssim \sum_{K {dyadic}} \frac{K^{1^-}}{K^{-2b'}}<\infty.
	\end{align}
Similarly, if one supposes that $w_{j\in A}\in L^1$ and $w_{j\notin A}\in L^\infty$,
	\begin{align}\label{eq:interp2}
	\left|\int_{\Gamma_\xi, \Gamma_\tau} \frac{M_2}{\prod_{j\notin A}\jap{\tau_j-\xi_j^3}^{2b}}w_1\dots w_k w \right| \lesssim \left(\prod_{j\in A}\|w_j\|_{L^1}\right)\left(\prod_{j\notin A}\|w_j\|_{L^\infty}\right).
\end{align}
By Stein's interpolation, we deduce \eqref{eq:versaoL2}.

\medskip
For \eqref{eq:multibourgKgrande}, since
$$
\tau-\xi^3 - \sum_{j=1}^k (\tau_j-\xi_j^3) = \Phi,
$$
we can suppose, without loss of generality, that $|\tau-\xi^3|\gtrsim |\Phi| > K$. Going back to \eqref{eq:integraltau}, we can refine the estimate as
	\begin{align}
	\int_{\Gamma_\tau} \frac{\jap{\tau-\xi^3}^{2b'}}{\prod_{j=1}^k\jap{\tau_j-\xi_j^3}^{2b}} d\tau_{j\in A} \lesssim 	\int_{\Gamma_\tau} \frac{\jap{\tau-\xi^3}^{2b'+0^+}K^{0^-}}{\prod_{j=1}^k\jap{\tau_j-\xi_j^3}^{2b}} d\tau_{j\in A} \lesssim  \frac{K^{0^-}}{\jap{\Phi - \sum_{j\notin A}(\tau_j-\xi_j^3)}^{-2b'+0^-}}
\end{align}
which gives the claimed factor $K^{0^-}$.

\medskip
Estimate \eqref{eq:multibourgKpeq} follows from dropping the $\Phi-\alpha$ weight in \eqref{eq:diadic} on both sides of the interpolation:
	\begin{align}
	\sup_{\xi_{j\notin A},\tau_{j\notin A}}\int_{\Gamma_\xi, \Gamma_\tau} \frac{M_1\mathbbm{1}_{|\Phi|<K}\jap{\tau-\xi^3}^{2b'}}{\prod_{j=1}^k\jap{\tau_j-\xi_j^3}^{2b}} d\xi_{j\in A} d\tau_{j\in A} &\lesssim 	\sup_{\xi_{j\notin A},\alpha}\int_{\Gamma_\xi} \frac{M_1\mathbbm{1}_{|\Phi|<K}}{\jap{\Phi-\alpha}^{-2b'}} d\xi_{j\in A} \\&\lesssim 	\sup_{\xi_{j\notin A},\alpha}\int_{\Gamma_\xi} M_1\mathbbm{1}_{|\Phi|<K} d\xi_{j\in A} \lesssim K^{1^-}.
\end{align}

\medskip
Finally, we consider \eqref{eq:multibourgPhi}, which is equivalent to 
\begin{equation}
	\left|\int_{\Gamma_\xi, \Gamma_\tau} \frac{M\jap{\xi}^s\jap{\tau-\xi^3}^{b'}\mathbbm{1}_{|\Phi|>K}}{\jap{\xi_1}^{s_1}\jap{\tau_1-\xi_1^3}^{b'}\prod_{j\neq0,1}\jap{\xi_j}^{s_j}\jap{\tau_j-\xi_j^3}^b|\Phi|}w_1\dots w_k w \right|\lesssim \prod_{j=1}^k \|w_j\|_{L^2}
\end{equation}
Let us consider the worst-case scenario, where the modulation variables are ordered as
$$
 |\tau_1-\xi_1^3| \ge  |\tau-\xi^3| \ge  |\tau_2-\xi_2^3| \ge \dots \ge |\tau_k-\xi_k^3|.
$$
If $|\tau_1-\xi_1^3|\sim |\tau-\xi^3|$, then we interpolate between the estimates
\begin{align*}
	\left|\int_{\Gamma_\xi, \Gamma_\tau} \frac{M_1\mathbbm{1}_{|\Phi|>K}}{\prod_{ j\neq 0,1}\jap{\tau_j-\xi_j^3}^{b}|\Phi|}w_1\dots w_k w \right| \lesssim\ \left(\prod_{j\in A}\|w_j\|_{L^\infty_\xi L^2_\tau}\right)\left(\prod_{j\notin A}\|w_j\|_{L^1_\xi L^2_\tau}\right)
\end{align*}
and
\begin{align*}
	\left|\int_{\Gamma_\xi, \Gamma_\tau} \frac{M_2\mathbbm{1}_{|\Phi|>K}}{\prod_{ j\neq 0,1}\jap{\tau_j-\xi_j^3}^{b}|\Phi|}w_1\dots w_k w \right| \lesssim\ \left(\prod_{j\in A}\|w_j\|_{L^1_\xi L^2_\tau}\right)\left(\prod_{j\notin A}\|w_j\|_{L^\infty_\xi L^2_\tau}\right)
\end{align*}
For the first,
	\begin{align*}
	&\	\left|\int_{\Gamma_\xi, \Gamma_\tau} \frac{M_1\mathbbm{1}_{|\Phi|>K}}{\prod_{ j\neq 0,1}\jap{\tau_j-\xi_j^3}^{b}|\Phi|}w_1\dots w_k w \right| \\\lesssim\ & \int_{\Gamma_\xi} \frac{M_1\mathbbm{1}_{|\Phi|>K}}{|\Phi|}\left(	\sup_{\tau}\int_{ \Gamma_\tau} \frac{1}{\prod_{j\neq 0,1}\jap{\tau_j-\xi_j^3}^{2b}} d\tau_2\dots d\tau_k\right)\prod_{j=0}^k \|w_j(\xi_j)\|_{L^2_\tau}\\\lesssim\ & \left[\sup_{\xi_{j\notin A}}\int_{\Gamma_\xi} \frac{M_1\mathbbm{1}_{|\Phi|>K}}{|\Phi|}\left(	\sup_{\tau}\int_{ \Gamma_\tau} \frac{1}{\prod_{j\neq 0,1}\jap{\tau_j-\xi_j^3}^{2b}} d\tau_2\dots d\tau_k\right)d\xi_{j\in A}\right]\left(\prod_{j\in A}\|w_j\|_{L^\infty_\xi L^2_\tau}\right)\left(\prod_{j\notin A}\|w_j\|_{L^1_\xi L^2_\tau}\right).
\end{align*}
Since $b>1/2$, the integral in the time frequencies is bounded and 
\begin{align*}
\sup_{\xi_{j\notin A}}\int_{\Gamma_\xi} \frac{M_1\mathbbm{1}_{|\Phi|>K}}{|\Phi|}\left(	\sup_{\tau}\int_{ \Gamma_\tau} \frac{1}{\prod_{j\neq 0,1}\jap{\tau_j-\xi_j^3}^{2b}} d\tau_2\dots d\tau_k\right)d\xi_{j\in A} &\lesssim \sup_{\xi_{j\notin A}}\int_{\Gamma_\xi} \frac{M_1\mathbbm{1}_{|\Phi|>K}}{|\Phi|} d\xi_{j\in A}\\ &\lesssim \sum_{K'>K} \frac{1}{K'}\sup_{\xi_{j\notin A}}\int_{\Gamma_\xi} M_1\mathbbm{1}_{|\Phi|<K'} d\xi_{j\in A} \\&\lesssim \sum_{K'>K} (K')^{0^-} \lesssim K^{0^-}.
\end{align*}
The other side of the interpolation is completely analogous. If $|\tau_1-\xi_1^3|<1$, the integrals in $\tau$ are still bounded and the same estimates hold. We are left with the case $|\tau_1-\xi_1^3|\gg \max\{1,|\tau-\xi^3|\}$. Consider the case where $1\in A$, as the other case follows from similar computations. We interpolate between
\begin{align}\label{eq:interp}
	\left|\int_{\Gamma_\xi, \Gamma_\tau} \frac{M_1\mathbbm{1}_{|\Phi|>K}\jap{\tau-\xi^3}^{2b'}}{\jap{\tau_1-\xi_1^3}^{2b'}\prod_{j\in A, j\neq 0,1}\jap{\tau_j-\xi_j^3}^{2b}|\Phi|^2}w_1\dots w_k w \right| \lesssim\ \left(\prod_{j\in A}\|w_j\|_{L^\infty}\right)\left(\prod_{j\notin A}\|w_j\|_{L^1}\right)
\end{align}
and \eqref{eq:interp2}. To prove \eqref{eq:interp}, notice that $|\tau_1-\xi_1^3|\sim |\Phi|$. Thus
\begin{align}
		&\left|\int_{\Gamma_\xi, \Gamma_\tau} \frac{M_1\mathbbm{1}_{|\Phi|>K}\jap{\tau-\xi^3}^{2b'}}{\jap{\tau_1-\xi_1^3}^{2b'}\prod_{j\in A, j\neq 0}\jap{\tau_j-\xi_j^3}^{2b}|\Phi|^2}w_1\dots w_k w \right|\\ \lesssim \ &	\int_{\Gamma_\xi, \Gamma_\tau} \frac{M_1\jap{\tau-\xi^3}^{2b'}K^{0^-}}{\jap{\tau_1-\xi_1^3}^{2b'+2+0^-}\prod_{j\in A, j\neq 0}\jap{\tau_j-\xi_j^3}^{2b}}|w_1\dots w_k w| 
\end{align}
and the proof now follows as done for \eqref{eq:interp3}. The proof of \eqref{eq:multibourgPhi2} follows from analogous computations.
\end{proof}


\begin{nb}
	The choice of $A$, $M_1$ and $M_2$ may depend on the relation between the various spatial frequencies. As long as the number of such regions is finite, one first decomposes onto each region, chooses $A$, $M_1$ and $M_2$ appropriately, applies Lemma \ref{lem:interpol} and then sums up the contribution from each region. For the sake of simplicity, we will often omit this detail.
\end{nb}

We now recall the frequency restricted estimates proved in \cite[Proof of Proposition 6]{COS23}:
\begin{lem}[Basic smoothing estimate]
	There exist $\emptyset\neq A \subsetneq \{0,\dots,k\}$, ${M}_n={M}_n(\xi_1,\dots,\xi_k)\ge 0$, $n=1,2$, such that
	$$
	\left({M}_1{M}_2\right)^{\frac{1}{2}}=\max_{j=1,\dots, 4}|\xi_j|\jap{\xi_j}^{(1/2)^-}
	$$
	and, for any $K>0$,
	\begin{equation}\label{eq:fre}
		\sup_{\xi_{j\notin A},\alpha} \int_{\Gamma_\xi}{M}_1\fia d\xi_{j\in A} + 	\sup_{\xi_{j\in A},\alpha} \int_{\Gamma_\xi}{M}_2\fia d\xi_{j\notin A}  \lesssim K^{1^-}.
	\end{equation}
In particular, by Lemma \ref{lem:interpol},
\begin{equation}\label{eq:nonlinearsmooth1}
	\||\nabla|u_1u_2u_3u_4\|_{X^{(1/2)^-,b'}}\lesssim \prod_{j=1}^4\|u_j\|_{X^{0,b}}.
\end{equation}
\end{lem}

\begin{lem}\label{lem:pointwise}
	On $\Gamma_5$, we have the pointwise estimate
	\begin{equation}\label{eq:pointwise}
		\left|\frac{m_1^2\xi_1+\dots +m_5^2\xi_5}{m_1\dots m_5}\right|\lesssim \frac{\max_{j=1,\dots,5}|\xi_j|\jap{\xi_j}^{(1/2)^-}}{N^{(1/2)^-}}.
	\end{equation}
	
\end{lem}
\begin{proof}
	For the sake of simplicity, assume that $|\xi_1|\ge \dots \ge |\xi_5|$. If all frequencies are smaller than $N$, then
	$$
	m_1^2\xi_1+\dots +m_5^2\xi_5 = \xi_1+\dots + \xi_5 =0.
	$$
	Otherwise, we have $|\xi_1|> N$, $m_1\le \dots \le m_5$ and $m_1^2\xi_1 \ge \dots \ge m_5^2\xi_5$. Thus
	\begin{equation}
		\left|\frac{m_1^2\xi_1+\dots +m_5^2\xi_5}{m_1\dots m_5}\right|\lesssim	\frac{m_1^2|\xi_1|}{m_1^5} \sim \frac{|\xi_1|^{1-3s}}{N^{-3s}}\lesssim \frac{|\xi_1|\jap{\xi_1}^{(1/2)^-}}{N^{(1/2)^-}}
	\end{equation}
	since $s>-1/6$.
\end{proof}

\begin{cor}\label{cor:gps}
Estimate \eqref{eq:N12_0} holds.
\end{cor}
\begin{proof}
Rewrite the estimate as
	\begin{equation}\label{eq:N12}
		\left\| \left(\frac{m_1^2\xi_1+\dots+m_5^2\xi_5}{\prod_{j=1}^5m_j}\right)^\vee v_1v_2v_3v_4 \right\|_{X^{0,b'}}\lesssim N^{(-1/2)^+}\prod_{j=1}^4\|v_j\|_{X^{0,b}},\quad \xi_5=-\xi_{1234}
	\end{equation}
	Assume that $|\xi_1|\ge \dots\ge |\xi_5|$. Since, by Lemma \ref{lem:pointwise},
	$$
\frac{m_1^2\xi_1+\dots+m_5^2\xi_5}{\prod_{j=1}^5m_j}\lesssim \frac{|\xi_1|\jap{\xi_1}^{(1/2)^-}}{N^{(1/2)^-}},
	$$
	\eqref{eq:N12} follows from \eqref{eq:nonlinearsmooth1} through Lemma \ref{lem:monot}.
\end{proof}

%

%


\begin{proof}[Proof of Lemma \ref{lem:defiE2}]
We need to prove that
$$
\left| \int_{\Gamma_5} \frac{m_1^2\xi_1+\dots + m_5^2\xi_5}{\Phi_5m_1\dots m_5}\mathbbm{1}_{|\Phi_5|>K}v_1\dots v_5  \right|\lesssim N^{(-1/2)^+}K^{0^-} \prod_{j=1}^5 \|v_j\|_{L^2}.
$$
This follows from decomposing dyadically in $K$ and from an iterpolation argument (see the proof of Lemma \ref{lem:interpol}) between estimates \eqref{eq:pointwise}:
\begin{align*}
	\left| \int_{\Gamma_5} \frac{m_1^2\xi_1+\dots + m_5^2\xi_5}{\Phi_5m_1\dots m_5}\mathbbm{1}_{|\Phi_5|>K}v_1\dots v_5  \right| &\lesssim \sum_{K'>K} \frac{1}{K'}\int_{\Gamma_5} \frac{|m_1^2\xi_1+\dots + m_5^2\xi_5|}{m_1\dots m_5}\mathbbm{1}_{|\Phi_5|\sim K'}|v_1|\dots |v_5| \\&\lesssim \sum_{K'>K} \frac{1}{K'}\int_{\Gamma_5} \frac{\max_{j=1,\dots,5}|\xi_j|\jap{\xi_j}^{(1/2)^-}}{N^{(1/2)^-}}\mathbbm{1}_{|\Phi_5|\sim K'}|v_1|\dots |v_5|\\& \lesssim N^{(-1/2)^+}K^{0^-} \prod_{j=1}^5 \|v_j\|_{L^2}.
\end{align*}
\end{proof}
\begin{proof}[Proof of Lemma \ref{lem:L5}]
	We want to prove that
	\begin{equation}
		\left|	\int_0^t \Lambda_5\left( \frac{m_1^2\xi_1+\dots m_5^2\xi_5}{m_1\dots m_5}\mathbbm{1}_{|\Phi_5|<K};v \right) ds \right| \lesssim t^{0^+}N^{(-1/2)^+}K^{1^-}\|v\|_{X^{0,1/2^+}}^5.
	\end{equation}
	By duality,
	\begin{equation}
		\left|	\int_0^t \Lambda_5\left( \frac{m_1^2\xi_1+\dots m_5^2\xi_5}{m_1\dots m_5}\mathbbm{1}_{|\Phi_5|<K};v \right) ds \right| \lesssim t^{0^+}\left\|\left(\frac{m_1^2\xi_1+\dots m_5^2\xi_5}{m_1\dots m_5}\mathbbm{1}_{|\Phi_5|<K}\right)^\vee v^4 \right\|_{X^{0,b'}}\|v\|_{X^{0,b}}
	\end{equation}
	The estimate follows from the pointwise multiplier bound \eqref{eq:pointwise}, the frequency-restricted estimate \eqref{eq:fre} and Lemma \ref{lem:interpol}. 
\end{proof}

\subsection{Proof of Lemma \ref{lem:L8}}\label{sec:multi2} We now focus on the proof of Lemma \ref{lem:L8}. In equivalent form,
$$
\left|\int_0^t \int \frac{\mathcal{M}}{\prod_{j=1}^8 m_j} \hat{v}_1\dots \hat{v}_8 d\Xi ds\right| \lesssim t^{0^+}N^{(-1)^+}K^{0^-}\prod_{j=1}^8 \|v_j\|_{X^{0,b}},$$
where
$$
 \mathcal{M}=\frac{\mathcal{M}_1(\xi_1,\dots,\xi_4,-\xi_{1234})}{\Phi_5}\mathbbm{1}_{|\Phi_5|>K}\xi_{1234},\quad \mathcal{M}_1=m_1^2\xi_1+\dots+m_5^2\xi_5.
$$
Without loss of generality, we can suppose that $|\xi_1|\ge \dots \ge |\xi_4|$ and $|\xi_5|\ge \dots \ge |\xi_8|$. Moreover, if $|\xi_1|\ll N$, then $\mathcal{M}=0$, and so we can assume that $|\xi_1|\gtrsim N$. Following the heuristic of Section \ref{sec:descript}, we split this domain into three regions:
\begin{align*}
D_1&=\{|\xi_5|\gtrsim N\}\\
D_2&=B\cap D_1^c, \quad B:=\{(\xi_1,\dots, \xi_4,\xi_{1234})\in \R^5: |\xi_1|\simeq |\xi_1|\simeq |\xi_2|\simeq |\xi_3| \gg |\xi_{1234}|\}^c\\
D_3&=D_1^c\cap D_2^c.
\end{align*}


\bigskip

\noindent\textbf{Control over $D_1$.} In this region, we can exploit the basic smoothing estimate twice.
We estimate
\begin{align}
	\left|\int_0^t \int_{D_1} \frac{\mathcal{M}}{\prod_{j=1}^8 m_j} \hat{v}_1\dots \hat{v}_8 d\Xi ds\right| &\lesssim t^{0^+}\left\| \frac{{\mathcal{M}_1}(\xi_1,\dots,\xi_4,-\xi_{1234})}{\Phi_5m_{1234}\prod_{j=1}^4m_j}\mathbbm{1}_{|\Phi_5|>K} v_1v_2v_3v_4  \right\|_{X^{0,b}}\\&\qquad \times\left\| \frac{\xi_{5678}m_{5678}}{\prod_{j=5}^8m_j}\mathbbm{1}_{D_1}v_5v_6v_7v_8\right\|_{X^{0,b'}}\label{eq:D1}
\end{align}
The control of these two factors is the content of the next lemma.
\begin{lem}\label{lem:D1}
	Given $K>0$,
	$$
	\left\| \left(\frac{\mathcal{M}_1}{|\Phi|m_{1234}\prod_{j=1}^4m_j}\mathbbm{1}_{|\Phi|>K}\right)^\vee v_1v_2v_3v_4\right\|_{X^{0,b}}\lesssim N^{-1/2^-}K^{0^-} \prod_{j=1}^4\|v_j\|_{X^{0,b}}
	$$
	and
	$$
	\left\| \left(\frac{\xi m}{\prod_{j=1}^4m_j}\mathbbm{1}_{\{\max |\xi_j|>N\} }\right)^\vee v_1v_2v_3v_4\right\|_{X^{0,b'}}\lesssim N^{-1/2^-} \prod_{j=1}^4\|v_j\|_{X^{0,b}}
	$$
\end{lem}
\begin{proof}
	Since, by Lemma \ref{lem:pointwise},
	$$
	\frac{\mathcal{M}_1}{m_{1234}\prod_{j=1}^4m_j}\lesssim \frac{\max_{j=1,\dots, 4} |\xi_j|\jap{\xi_j}^{(1/2)^-}}{N^{(1/2)^-}},
	$$
	the first estimate follows from the frequency-restricted estimate \eqref{eq:fre}, \eqref{eq:multibourgPhi2} and Lemma \ref{lem:monot}. For the second bound, observe (see the proof of Lemma \ref{lem:pointwise}) that 
	 $$
	 \frac{\xi m}{\prod_{j=1}^4m_j}\mathbbm{1}_{\{\max |\xi_j|>N\} } \lesssim \frac{\max_{j=1,\dots, 4} |\xi_j|\jap{\xi_j}^{(1/2)^-}}{N^{(1/2)^-}}.
	 $$
	 It now suffices to use \eqref{eq:fre} in Lemma \ref{lem:interpol}.
\end{proof}

\bigskip

\noindent\textbf{Control over $D_2$.} Here, we use a refined smoothing effect. Write
	$$
	V=\frac{\xi_{5678}\jap{\xi_{5678}}^{(1/2)^-}m_{5678}}{\prod_{j=5}^8m_j}\hat{v}_5\dots \hat{v}_8=\xi_{5678}\jap{\xi_{5678}}^{(1/2)^-}\hat{v}_5\dots \hat{v}_8
	$$
	Observe that, by the basic smoothing estimate,
	$$
	\|V\|_{X^{0,b'}} \lesssim \prod_{j=5}^8\|v_j\|_{X^{0,b}}.
	$$
	Then
\begin{align}
	\left|\int_0^t \int_{D_2} \frac{\mathcal{M}}{\prod_{j=1}^8 m_j} \hat{v}_1\dots \hat{v}_8 d\Xi ds\right|&= \left|\int_0^t \int_{D_2} \frac{\mathcal{M}_1}{\Phi_5\jap{\xi_{1234}}^{(1/2)^-}\prod_{j=1}^5 m_j} \hat{v}_1\dots \hat{v}_4V d\Xi ds\right| \\& \lesssim \left\| v_1  \right\|_{X^{0,b}}\left\| \left(\frac{{\mathcal{M}_1}(\xi_1,\dots,\xi_4,-\xi_{1234})}{\Phi_5\jap{\xi_{1234}}^{(1/2)^-}m_{1234}\prod_{j=1}^4m_j}\mathbbm{1}_{|\Phi_5|>K}\mathbbm{1}_B\right)^\vee v_2v_3v_4V\right\|_{X^{0,b'}}\label{eq:D2}
\end{align}
The last factor can be shown to be bounded:
\begin{lem}[Refined smoothing over $B$]\label{lem:refined}
	$$
	\left\|\left( \frac{{\mathcal{M}_1}(\xi_1,\dots,\xi_4,-\xi_{1234})}{\Phi_5\jap{\xi_{1234}}^{(1/2)^-}m_{1234}\prod_{j=1}^4m_j}\mathbbm{1}_{|\Phi_5|>K}\mathbbm{1}_B\right)^\vee v_1v_2v_3v_4\right\|_{X^{0,b'}}\lesssim N^{(-1)^+}K^{0^-} \left(\prod_{j=1}^3\|v_j\|_{X^{0,b}}\right)\|v_4\|_{X^{0,b'}}.
	$$ 
\end{lem}
\begin{proof}
		
	Without loss of generality, we may assume that $|\xi|\ge |\xi_1|\ge \dots \ge |\xi_4|$.
	Since, by Lemma \ref{lem:pointwise},
	$$
	\frac{{\mathcal{M}_1}(\xi_1,\dots,\xi_4,-\xi_{1234})}{\jap{\xi_{1234}}^{(1/2)^-}m_{1234}\prod_{j=1}^4m_j} \lesssim \frac{|\xi|\jap{\xi}^{(1/2)^-}}{\jap{\xi_4}^{(1/2)^-}N^{(1/2)^-}}\lesssim \frac{|\xi|\jap{\xi}^{(1^-)^-}}{\jap{\xi_4}^{(1/2)^-}N^{(1^-)^-}}
	$$
	(with $2(1/2)^->(1^-)^-$), the proof will follow from \eqref{eq:multibourgPhi}, after we show the necessary frequency-restricted estimate for the multiplier
	$$
	M=\frac{|\xi|\jap{\xi}^{(1^-)^-}}{\jap{\xi_4}^{(1/2)^-}}
	$$
	We will follow closely the arguments of \cite[Proposition 6]{COS23}.
	\medskip
	
	\noindent\textbf{Case A.} $|\xi|\lesssim |\xi_4|$. We interpolate between
	$$
	\sup_{\xi,\xi_3} \int \frac{|\xi|\jap{\xi}^{(1^-)^-}}{\jap{\xi_4}^{1^-}}\mathbbm{1}_{|\Phi_5|\sim K} d\xi_1d\xi_2 =: \sup_{\xi,\xi_3} I_1
	$$
	and
	$$
	\sup_{\xi_1,\xi_2,\xi_4} \int |\xi|\jap{\xi}^{(1^-)^-}\mathbbm{1}_{|\Phi_5|\sim K} d\xi.=: \sup_{\xi_1,\xi_2,\xi_4} I_2
	$$
	We begin with $I_1$. In what follows, we take $p=1^+$.

	\noindent \textbf{a)} $|\partial_{\xi_1}\Phi_5|\gtrsim|\xi|^2$. This allows for the change of variables $\xi_1\mapsto \Phi_5$:
	\begin{align*}
		I_1 &\lesssim \left[ \int \left(\frac{|\xi|\jap{\xi}^{(1^-)^-}}{\jap{\xi_4}^{1^-}}\right)^p  \mathbbm{1}_{|\Phi_5|\sim K} d\xi_1d\xi_2\right]^{\frac{1}{p}}|\xi|^{2/p'}\\&\lesssim \left[ \int \left(|\xi|^{1^-}\right)^p  \mathbbm{1}_{|\Phi_5|\sim K} \frac{1}{|\xi|^2}d\Phi_5d\xi_2\right]^{\frac{1}{p}}|\xi|^{2/p'}  \lesssim |\xi|^{2/p'}\left[ |\xi|^{p^--1}K \right]^{1/p}\lesssim K^{1^-}.
	\end{align*}
	\noindent\textbf{b)} $|\partial_{\xi_1}\Phi_5|\ll |\xi|^2$, which means that $|\xi_1|\simeq |\xi_4|$. We renormalize the frequencies $p_j=\xi_j/\xi$ and write
	$$
	\Phi_5=\xi^3P, \quad P=1-\sum_{j=1}^4 p_j^3.
	$$
	A direct computation shows that $P$ has no degenerate stationary points. By Morse's lemma, this implies the existence of a local change of coordinates $(p_1,p_2)\mapsto (q_1,q_2)$ such that, for some $P_0$ independent of $q_1,q_2$,
	$$
	P=P_0 \pm q_1^2 \pm q_2^2.
	$$
	Then
	\begin{align*}
		I_1 &\lesssim \lesssim \left[ \int \left(\frac{|\xi|\jap{\xi}^{(1^-)^-}}{\jap{\xi_4}^{1^-}}\right)^p  \mathbbm{1}_{|\Phi_5|\sim K} d\xi_1d\xi_2\right]^{\frac{1}{p}}|\xi|^{2/p'}\\&\lesssim \left[ \int_{|q_1|,|q_2|\ll 1} |\xi|^{p^- + 2} \mathbbm{1}_{|\xi^3(P_0\pm q_1^2\pm q_2^2)|\sim K} |\xi|^2dq_1dq_2\right]^{\frac{1}{p}}|\xi|^{2/p'} \lesssim |\xi|^{2/p'}\left[ |\xi|^{p(1^-)^--1}K^{1^-} \right]^{1/p}\lesssim K^{1^-}.
	\end{align*}
	We now move to $I_2$. Observe that, if $|\partial_\xi \Phi_5|\ll |\xi_1|^2$, then $|\xi|\simeq \dots \simeq |\xi_3|\gg |\xi_4|$, contradicting the definition of $B$. Then $|\partial_\xi \Phi_5|\gtrsim |\xi_1|^2$ and the change of variables $\xi\mapsto \Phi_5$ yields
	\begin{align*}
		I_2\lesssim \left[\int |\xi_1|^{(1+(1^-)^-)p}\mathbbm{1}_{|\Phi_5|\sim K} \frac{d\Phi_5}{|\xi_1|^2} \right]^{1/p} |\xi_1|^{1/p'}\lesssim K^{1^-}.
	\end{align*}
	
	\medskip
	\noindent\textbf{Case B.} $|\xi_4|\ll |\xi|$. By the definition of $B$, $|\xi|\not\simeq |\xi_3|$.
	
	\noindent\textbf{a) }$|\xi_1|\not\simeq |\xi_2|$. Then we interpolate between
	$$
	\sup_{\xi,\xi_3} \int \frac{|\xi|^{1+(1^-)^-}}{\jap{\xi_4}^{1^-}}\mathbbm{1}_{|\Phi_5|\sim K}d\xi_1 d\xi_4 \lesssim |\xi|^{1/p'}\int \frac{1}{\jap{\xi_4}^{1^-}} \left[\int |\xi|^{p(1+(1^{-})^-)} \mathbbm{1}_{|\Phi_5|\sim K} \frac{d\Phi_5}{|\xi|^2}\right]^{1/p}d\xi_4 \lesssim K^{1^-}.
	$$
	and
	$$
	\sup_{\xi_1,\xi_2,\xi_4} \int |\xi|^{1+(1^-)^-}\mathbbm{1}_{|\Phi_5|\sim K}d\xi \lesssim |\xi_1|^{1/p'} \left[  \int  |\xi|^{p(1+(1^-)^-)}\mathbbm{1}_{|\Phi_5|\sim K} \frac{d\Phi_5}{|\xi_1|^2}\right]^{1/p} \lesssim K^{1^-}.
	$$
	\noindent\textbf{b) }$|\xi_1|\simeq |\xi_2|$ and $|\xi_1|\not\simeq |\xi_3|$. This implies that $|\xi|\not\simeq |\xi_2|$. We then proceed as in subcase a), exchanging $\xi$ with $\xi_1$.
	
	\noindent\textbf{c) }$|\xi_1|\simeq |\xi_2|\simeq |\xi_3|$. This implies that $|\xi_1|\simeq |\xi|/3$ and that $|\Phi_5|\sim |\xi|^3$. We interpolate between
	$$
	\sup_{\xi,\xi_3} \int \frac{|\xi|^{(1+(1^-)^-)}}{\jap{\xi_4}^{1^-}} \mathbbm{1}_{|\xi^3|\sim K} d\xi_1d\xi_4 \lesssim K^{1^-}\ln \jap{K}.
	$$
	and
	$$
	\sup_{\xi_1,\xi_2,\xi_4} \int |\xi|^{1+(1^-)^-}\mathbbm{1}_{|\xi|^3\sim K}d\xi\lesssim K^{1^-}.
	$$

\end{proof}

\bigskip

\noindent\textbf{Control over $D_3$.} This is the most nontrivial region, as it cannot be reduced to two quintic estimates (as we have done in the previous regions). For this region, we need to prove a full 8-linear estimate:

\begin{lem}\label{lem:D3}
	$$
	\left\| \left(\frac{\mathcal{M}}{m\prod_{j=1}^7m_j}\mathbbm{1}_{D_3}\right)^\vee u_1\dots u_7 \right\|_{X^{0,b'}} \lesssim N^{(-1)^+}K^{0^-}\prod_{j=1}^7\|u_j\|_{X^{0,b}}.
	$$
\end{lem}

\begin{proof}
	First, we recall that, over $D_3$,
	$$
	|\xi|\simeq |\xi_1| \simeq |\xi_2| \simeq |\xi_3| \gtrsim N\gg |\xi_4|>\dots>|\xi_7|, 
	$$
	so that
	$$
	\left|\frac{m^2\xi + m_1^2\xi_1 + \dots +m_3^2\xi_3 +m_{4567}^2\xi_{4567}}{m\prod_{j=1}^7m_j}\xi_{4567}\right| \lesssim \frac{|\xi_2||\xi_4|}{N^{-4s}|\xi|^{4s}} \lesssim \frac{|\xi|^{1-4s}|\xi_4|}{N^{-4s}}.
	$$
	Then the claimed estimate will follow from \eqref{eq:multibourg}, once we prove the frequency-restricted estimates associated to the multiplier
	$$
	M=\frac{|\xi|^{1-4s}|\xi_4|}{N^{-4s}|\Phi_5|}\mathbbm{1}_{|\Phi_5|>K}.
	$$
	By decomposing dyadically in $|\Phi_5|$, we can further reduce our problem to the multiplier
		$$
	M=\frac{|\xi|^{1-4s}|\xi_4|}{N^{-4s}}\mathbbm{1}_{|\Phi_5|\sim K}.
	$$
	
	For the sake of clarity, we show how to obtain the estimates with a full power of $K$ and $K'$. The exponent can be reduced to $1^-$ through the application of Hölder's inequality (as it was done in the proof of Lemma \ref{lem:refined}).

	\noindent\textbf{Case A.} $|\xi_4^2-\xi_{4567}^2|\gtrsim |\xi_4|^2$ and $|\xi_7^2-\xi_{4567}^2|\gtrsim |\xi_4|^2$. We interpolate between
	$$
	\sup_{\xi,\xi_2,\xi_4,\xi_6}\int \frac{|\xi|^{1-4s}|\xi_4|}{N^{-4s}} \mathbbm{1}_{|\Phi_5|\sim K}\mathbbm{1}_{|\Phi_8-\alpha|<K'}d\xi_1 d\xi_3 d\xi_5
	$$
	and
		$$
	\sup_{\xi_1,\xi_3,\xi_5,\xi_7}\int \frac{|\xi|^{1-4s}|\xi_4|}{N^{-4s}} \mathbbm{1}_{|\Phi_5|\sim K}\mathbbm{1}_{|\Phi_8-\alpha|<K'}d\xi d\xi_2 d\xi_6.
	$$
	For the first,
	$$
	\left|\frac{\partial(\Phi_5,\Phi_8)}{\partial(\xi_1,\xi_3)}\right| = \left|\begin{matrix}
		\xi_1^2-\xi_{4567}^2 & \xi_3^2 - \xi_{4567}^2\\
		\xi_1^2-\xi_7^2 & \xi_3^2-\xi_7^2
	\end{matrix}\right| \simeq |\xi_1|^2|\xi_7^2-\xi_{4567}^2|\gtrsim |\xi|^2|\xi_4|^2
	$$
	and thus
	\begin{align*}
		& 	\sup_{\xi,\xi_2,\xi_4,\xi_6}\int\frac{|\xi|^{1-4s}|\xi_4|}{N^{-4s}}\mathbbm{1}_{|\Phi_5|\sim K}\mathbbm{1}_{|\Phi_8-\alpha|<K'}d\xi_1 d\xi_3 d\xi_5 \\ \lesssim & \sup_{\xi,\xi_2,\xi_4,\xi_6} \int \frac{|\xi|^{1-4s}|\xi_4|}{N^{-4s}} \mathbbm{1}_{|\Phi_5|\sim K}\mathbbm{1}_{|\Phi_8-\alpha|<K'} \frac{1}{|\xi|^2|\xi_4|^2}d\Phi_5 d\Phi_8 d\xi_5\\\lesssim &\sup_{\xi,\xi_2,\xi_4,\xi_6}  \frac{|\xi|^{-1-4s}}{N^{-4s}|\xi_4|} \int d\xi_5 \cdot KK' \lesssim \frac{KK'}{N}.
	\end{align*}
For the second, a similar computation holds:
	$$
\left|\frac{\partial(\Phi_5,\Phi_8)}{\partial(\xi,\xi_2)}\right| = \left|\begin{matrix}
	\xi^2-\xi_{4567}^2 & \xi_2^2 - \xi_{4567}^2\\
	\xi^2-\xi_4^2 & \xi_2^2-\xi_4^2
\end{matrix}\right| \simeq |\xi|^2|\xi_4^2-\xi_{4567}^2|\gtrsim |\xi_1|^2|\xi_4|^2,
$$
from which
\begin{align*}
	& 	\sup_{\xi_1,\xi_3,\xi_5,\xi_7}\int \frac{|\xi|^{1-4s}|\xi_4|}{N^{-4s}} \mathbbm{1}_{|\Phi_5|\sim K}\mathbbm{1}_{|\Phi_8-\alpha|<K'}d\xi d\xi_2 d\xi_6 \\ \lesssim & \sup_{\xi_1,\xi_3,\xi_5,\xi_7} \int \frac{|\xi_1|^{1-4s}|\xi_4|}{N^{-4s}} \mathbbm{1}_{|\Phi_5|\sim K}\mathbbm{1}_{|\Phi_8-\alpha|<K'} \frac{1}{|\xi_1|^2|\xi_4|^2}d\Phi_5 d\Phi_8 d\xi_6\\\lesssim & \sup_{\xi_1,\xi_3,\xi_5,\xi_7} \frac{|\xi_1|^{-1-4s}}{N^{-4s}} \int \frac{1}{|\xi_5|}d\xi_6 \cdot KK' \lesssim \frac{KK'}{N^{1^-}}.
\end{align*}
\textbf{Case B.} $|\xi_4^2-\xi_{4567}^2|\ll |\xi_4|^2$ or $|\xi_7^2-\xi_{4567}^2|\ll |\xi_4|^2$. We interpolate between
	$$
\sup_{\xi,\xi_4,\xi_6,\xi_7}\int\frac{|\xi|^{1-4s}|\xi_4|}{N^{-4s}} \mathbbm{1}_{|\Phi_5|\sim K}\mathbbm{1}_{|\Phi_8-\alpha|<K'}d\xi_1 d\xi_2 d\xi_5
$$
and
$$
\sup_{\xi_1,\xi_2,\xi_3,\xi_5}\int \frac{|\xi|^{1-4s}|\xi_4|}{N^{-4s}} \mathbbm{1}_{|\Phi_5|\sim K}\mathbbm{1}_{|\Phi_8-\alpha|<K'}d\xi d\xi_6 d\xi_7.
$$
For the first, $\Phi_5$ is stationary in $\xi_1,\xi_2$ ($|\nabla_{\xi_1,\xi_3}\Phi_5|\ll |\xi|^2$) and does not depend on $\xi_5$. Moreover, $\Phi_8$ is not stationary in $\xi_5$. We normalize
$$
\xi^3-\xi_1^3-\dots-\xi_3^3=\xi^3P(p_1,p_2)=\xi^3(1-p_1^3-p_2^3-p_3^3),\quad p_j=\frac{\xi_j}{\xi}.
$$
Then $(p_1,p_2)$ is near one of the three possible stationary points
$$
(1,1)\quad (1,-1),\quad \mbox{and}\quad (-1,1).
$$
A direct verification shows that $\det(D^2P)\neq 0$ at each of these points. Applying Morse's lemma, there exists a local change of variables $(p_1,p_2)\mapsto(q_1,q_2)$ such that
$$
P(p_1,p_2)=P_0\pm q_1^2\pm q_2^2.
$$
This allows for the following computation:

\begin{align*}
	& \sup_{\xi,\xi_4,\xi_6,\xi_7}\int\frac{|\xi|^{1-4s}|\xi_4|}{N^{-4s}} \mathbbm{1}_{|\Phi_5|\sim K}\mathbbm{1}_{|\Phi_8-\alpha|<K'}d\xi_1 d\xi_2 d\xi_5\\ \lesssim \ & 	
	\sup_{\xi,\xi_4,\xi_6,\xi_7}\int\frac{|\xi|^{1-4s}|\xi_4|}{N^{-4s}|\xi|^2}\mathbbm{1}_{|\xi^3P|\sim K}\mathbbm{1}_{|\Phi_8-\alpha|<K'}|\xi|^2dp_1 dp_2 d\Phi_8\\ \lesssim \ & 	
	\sup_{\xi,\xi_4,\xi_6,\xi_7}\int\frac{|\xi|^{2-4s}}{N^{-4s}}\mathbbm{1}_{|\xi^3(P_0\pm q_1^2\pm q_2^2)|\sim K}\mathbbm{1}_{|\Phi_8-\alpha|<K'}dq_1 dq_2 d\Phi_8 \\ \lesssim \ & 	
	\sup_{\xi,\xi_4,\xi_6,\xi_7}\frac{|\xi|^{-1-4s+0^+}}{N^{-4s}}K^{1^-}K'\lesssim \frac{K^{1^-}K'}{N^{1^-}}
\end{align*}
For the other side of the interpolation, we compute
$$
\left|\frac{\partial(\Phi_5,\Phi_8)}{\partial(\xi,\xi_7)} \right|=\left|\begin{matrix}
	\xi^2-\xi_{4567}^2 & 0\\
	\xi^2-\xi_7^7 & \xi_4^2-\xi_7^2
\end{matrix}\right| = |\xi^2||\xi_4^2-\xi_7^2|.
$$
If $|\xi_4^2-\xi_7^2|\ll |\xi_4|^2$, then $|\xi_4|\simeq |\xi_7|$. Since the frequencies are ordered and we're in Case B, this would imply
$$
|\xi_4|\simeq|\xi_5|\simeq|\xi_6|\simeq|\xi_7|\simeq|\xi_{4567}|,
$$
which is impossible, as these frequencies add up to zero. We conclude that $|\xi_4^2-\xi_7^2|\gtrsim |\xi_4|^2$ and then
\begin{align*}
	&\sup_{\xi_1,\xi_2,\xi_3,\xi_5}\int \frac{|\xi|^{1-4s}|\xi_4|}{N^{-4s}} \mathbbm{1}_{|\Phi_5|\sim K}\mathbbm{1}_{|\Phi_8-\alpha|<K'}d\xi d\xi_6 d\xi_7\\ \lesssim & \sup_{\xi_1,\xi_2,\xi_3,\xi_5}\int \frac{|\xi|^{1-4s}|\xi_4|}{N^{-4s}}\mathbbm{1}_{|\Phi_5|\sim K}\mathbbm{1}_{|\Phi_8-\alpha|<K'}\frac{1}{|\xi|^2|\xi_4|^2}d\Phi_5 d\Phi_8d\xi_6.\\ \lesssim & \sup_{\xi_1,\xi_2,\xi_3,\xi_5}\int \frac{|\xi|^{-1-4s}}{N^{-4s}|\xi_5|}\mathbbm{1}_{|\Phi_5|\sim K}\mathbbm{1}_{|\Phi_8-\alpha|<K'}d\Phi_5 d\Phi_8d\xi_6.\\ \lesssim & \sup_{\xi_1,\xi_2,\xi_3,\xi_5} \frac{KK'|\xi|^{-1-4s}}{N^{-4s}}\lesssim \frac{KK'}{N}.
\end{align*}

\end{proof}

\begin{proof}[Proof of Lemma \ref{lem:L8}]
Decomposing the frequency domain into $D_1, D_2$ and $D_3$. For the $D_1$, we apply Lemma \ref{lem:D1} in \eqref{eq:D1}. The estimate over $D_2$ follows from \eqref{eq:D2} and Lemma \ref{lem:refined}. Finally, by Lemma \ref{lem:D3},
\begin{align*}
	&\left| \int_0^t \Lambda_8\left(\frac{m_1^2\xi_1+\dots+m_4^2\xi_4-m_{1234}^2\xi_{1234}}{\Phi_5}\xi_{5678}\mathbbm{1}_{|\Phi_5|>K}\mathbbm{1}_{D_3};u \right) ds \right|\\\lesssim \ &	t^{0^+}\left\| \left(\frac{\mathcal{M}}{m\prod_{j=1}^7m_j}\mathbbm{1}_{D_3}\right)^\vee Iu_1\dots Iu_7 \right\|_{X^{0,b'}}\|Iu\|_{X^{0,b}}\lesssim t^{0^+}N^{(-1)^+}K^{0^-}\|Iu\|_{X^{0,(1/2)^+}}^8.
\end{align*}

\end{proof}

\section{Proof of Theorem \ref{thm:main}}\label{sec:thm}

\begin{proof}[Proof of Theorem \ref{thm:main}]
	
	\textit{Step 1. Local existence for the modified equation.} For $N$ fixed, we rewrite \eqref{4kdv} in the equivalent form
	\begin{equation}\label{eq:equI}
		\partial_t Iu + \partial_x^3 Iu = I\partial_x (u^4),
	\end{equation}
	with initial data $Iu_0\in L^2(\R)$. As in the proof of \cite[Lemma 1]{GPS}, given initial data satisfying $\|Iu_0\|_{L^2}<2\epsilon_0$, the solution to \eqref{eq:equI} is defined over the time interval $[0,1]$ and its $X^{s,b}$ norm, restricted to the time interval $[0,1]$, satisfies $$\|Iu\|_{X_{[0,1]}^{0,(1/2)^+}}<4\epsilon_0.$$ Moreover, $\epsilon_0$ is independent\footnote{These facts required only $s>-1/6$ and we can reuse them here.} of $N$. Without loss of generality, we can take $\epsilon_0$ so that
	$$
	(4\epsilon_0)^3<\frac{1}{100},
	$$
	so that, by Lemma \ref{lem:defiE2}, $E_2(u)\simeq \|Iu\|_{L^2}$ for $\|Iu\|_{L^2}<2\epsilon_0$.
	
	\textit{Step 2. Iteration.}  For a fixed time $T\gg1$, let us consider $\rho\gg 1, \lambda$ and $N$ so that
	$$
	\rho^{-\frac{1}{6}-s}\|u_0\|_{H^s}<\epsilon_0,\quad \lambda=\rho N^{\frac{-6s}{1+6s}},\quad  N^{\frac{1^-+24s}{1+6s}}=\rho^3T.
	$$ 
	Since $s>-1/24$, we have $N \gg 1$. Then $N^{1^-}=\lambda^3T$ and $u_0^\lambda(x)=\lambda^{-2/3}u_0(x/\lambda)$ satisfies
	$$
	\|Iu_0^\lambda\|_{L^2}<\epsilon_0.
	$$
	Then the successive aplication of Step 1 gives a solution $u^\lambda$ defined over $[0,n]$, $n\in\mathbb{N}$, for as long as
	\begin{equation}\label{eq:coniter}
		\|Iu^\lambda(n)\|_{L^2}<2\epsilon_0.
	\end{equation}
	By the almost-conservation of the second modified energy \eqref{eq:almosconserved},
	\begin{align*}
		E_2(u^\lambda(n)) &\le E_2(u^\lambda(n-1)) + N^{(-1)^+}((4\epsilon_0)^5 + (4\epsilon_0)^8) \\&\le E_2(u^\lambda_0) + nN^{(-1)^+}((4\epsilon_0)^5 + (4\epsilon_0)^8)\\&\le E_2(u^\lambda_0) + nN^{(-1)^+}
	\end{align*}
	and thus condition \eqref{eq:coniter} is satisfied for $n<N^{1^-}$. By the definition of $N$, this means that $u^\lambda$ is defined over the time interval $[0,\lambda^3T]$. By rescaling, the solution $u$ to \eqref{4kdv} with initial data $u_0$ is defined over $[0,T]$. As $T$ is arbitrary, the solution is global-in-time. The growth bound follows as in \cite[Theorem 1]{GPS}.
\end{proof}
\bibliography{biblio}
\bibliographystyle{plain}

\end{document}